\documentclass[letterpaper, 10 pt, conference]{ieeeconf}  

\usepackage{amsmath,amssymb,amsfonts}
\usepackage{bm}
\usepackage{cite}
\usepackage{algorithm}
\usepackage{algorithmic}
\usepackage{graphicx}
\usepackage{caption}
\usepackage{lipsum}
\usepackage{enumerate}
\usepackage{amsthm}
\usepackage{hyperref}
\usepackage{cleveref}

\usepackage{amsmath}
\usepackage{amssymb}
\usepackage{color}
\usepackage{booktabs}

\newtheorem{definition}{Definition}
\newtheorem{lemma}{Lemma}
\newtheorem{theorem}{Theorem}
\newtheorem{proposition}{Proposition}

\usepackage{bm}
\usepackage{threeparttable}
 
\IEEEoverridecommandlockouts                              

\title{\LARGE \bf
Optimization Landscape of Gradient Descent for Discrete-time Static Output Feedback
}

\author{Jingliang Duan, Jie Li, Shengbo Eben Li, Lin Zhao
\thanks{J. Duan and L. Zhao  are with the Department of Electrical and Computer Engineering, National University of Singapore, Singapore. {\tt\small Email: (duanjl,elezhli)@nus.edu.sg}.
}
\thanks{J. Li and S. E. Li are with the School of Vehicle and Mobility, Tsinghua University, Beijing, 100084, China. {\tt\small Email: jie-li18@mails.tsinghua.edu.cn, lishbo@mail.tsinghua.edu.cn}.
}
}

\begin{document}

\maketitle
\thispagestyle{empty}
\pagestyle{empty}

\begin{abstract}
In this paper, we analyze the optimization landscape of gradient descent methods for static output feedback (SOF) control of discrete-time linear time-invariant systems with quadratic cost. The SOF setting can be quite common, for example, when there are unmodeled hidden states in the underlying process. We first establish several important properties of the SOF cost function, including coercivity, $L$-smoothness, and $M$-Lipschitz continuous Hessian. We then utilize these properties to show that the gradient descent is able to converge to a stationary point at a dimension-free rate. Furthermore, we prove that under some mild conditions, gradient descent  converges linearly to a local minimum if the starting point is close to one. These results not only characterize the performance of gradient descent in optimizing the SOF problem, but also shed light on the efficiency of general policy gradient methods in reinforcement learning.
\end{abstract}


%
\IEEEpeerreviewmaketitle

\section{Introduction}
%
%
%
%

Reinforcement learning (RL) has achieved human-level performance in a variety of challenging domains, from games to robotic control \cite{mnih2015DQN,silver2016mastering,lillicrap2015DDPG, duan2021distributional,duan2020hierarchical}.  For continuous control settings, policy gradient based RL algorithms such as DDPG \cite{lillicrap2015DDPG}, TRPO\cite{schulman2015TRPO}, PPO\cite{schulman2017PPO}, SAC\cite{Haarnoja2018SAC}, DSAC\cite{duan2021distributional}, are widely used to find the nearly optimal parameterized policy function. Despite their success in various applications, the theoretical understandings of their performance and computational complexity are rather limited, even in the most basic settings such as linear systems. 

Recent literature investigated the optimization landscape and the global convergence of policy gradient methods when applied to the linear quadratic regulator (LQR). It is well known that directly optimizing the LQR cost with respect to the linear policy is generally a nonconvex optimization. The domain of stabilizing linear policy can be nonconvex with non-smooth boundaries. As a pioneering work, Fazel \emph{et al.} (2018) demonstrated the gradient dominance and the almost smoothness properties of the LQR cost in a discrete-time setting, leading to the global convergence at a linear rate for policy gradient methods\cite{fazel2018global}.  This work has inspired a line of subsequent studies that examined similar properties through the lens of policy gradient methods for discrete-time and continuous-time LQR \cite{bu2019lqr,bhandari2019global,mohammadi2019CT-LQR,fatkhullin2020CTSOF}, finite-horizon noisy LQR \cite{hambly2020noisyLQR}, LQR tracking \cite{ren2021lqrtracking},  Markovian jump LQR \cite{jansch2020Mjump}, distributed LQR \cite{furieri2020distributedLQR}, linear $H_2$ control with $H_{\infty}$ constraints \cite{zhang2020robust}, as well as linear quadratic Gaussian problem \cite{tang2021LQG,zheng2021sampleLQG}.

The aforementioned literature mainly focuses on the full state feedback or dynamic output feedback control, whereas the convergence of the gradient method for static output feedback (SOF) LQR is barely studied. SOF is one of the most important open questions in control theory, which considers a linear time-invariant system with quadratic performance in
the case when only some outputs (linear functions of the state) instead of the complete state are available \cite{syrmos1997SOF}. Note that state feedback LQR can be regarded as a special case of SOF. Compared with state feedback LQR, the optimization landscape for SOF LQR is much more complicated. The domain of stabilizing SOF control gains can be disconnected and the stationary points in each component of the domain can be local minima, saddle points, or even local maxima \cite{fatkhullin2020CTSOF,feng2020connectivity,bu2019topological,gryazina2008d}. The recent work \cite{fatkhullin2020CTSOF} analyzed the convergence rate to stationary points, which however only focused on continuous-time linear systems. 

This paper aims to analyze the optimization landscape of the discrete-time SOF. The contributions are summarized as follows.
\begin{enumerate}
    \item We establish several important properties of the SOF cost function, including coercivity, $L$-smoothness, and $M$-Lipschitz continuous Hessian. These results enable further convergence analysis of policy gradient for discrete-time SOF.
    \item  For the partially observed case, results on convergence to stationary points and the corresponding dimension-free convergence rate are formally established. Moreover, utilizing the Lipschitz continuity of  Hessian, we demonstrate that under mild conditions, gradient descent converges linearly to a local minimum if the starting point is close to one.
    \item  For the fully observed case, we show that gradient descent converges to the global optimal controller at a linear rate if an initial stabilizing controller is used. Compared with the results given in \cite{fazel2018global}, we provide a more explicit characterization of computational complexity. 
\end{enumerate}

The remainder of this paper is organized as follows. In Section \ref{sec:preliminary}, we describe some preliminaries of SOF. Section \ref{sec.gradient_property} derives the formula of the gradient and establishes several key properties of the SOF cost function. Section \ref{sec:convergence} shows the convergence results of gradient descent for SOF, and Section \ref{sec:conclusion} concludes this article.

\textbf{Notation:}
$\|X\|$ and $\rho(X)$ denote the spectral norm and spectral radius of a matrix $X$, respectively; ${\rm Tr}(X)$, $\sigma_{\rm min}(X)$, and $\lambda_{\rm min}(X)$ denote the trace, minimal singular value, and minimal eigenvalue of a square matrix, respectively. $\partial \mathbb{X}$ denotes the boundary of a set $\mathbb{X}$. $\mathrm{vec}(X)$ denotes the vectorization of  matrix $X$. $X \succ Y$ and $X \succeq Y$ represent that $X-Y$ is positive definite and positive semidefinite, respectively.

\section{Preliminaries}
\label{sec:preliminary}
Consider the discrete-time time-invariant linear system 
\begin{equation}   
\label{eq.statefunction}
\begin{aligned}
x_{t+1} &= Ax_t+Bu_t,\\
y_t &= Cx_t,
\end{aligned}
\end{equation}
with state $x$, output $y$, and matrices $A\in \mathbb{R}^{n\times n}$, $B\in \mathbb{R}^{n\times m}$, $C\in \mathbb{R}^{d\times n}$. Without loss of generality, we assume $C$ has full row rank. The objective of infinite-horizon LQR performance is given by 
\begin{equation}   
\label{eq.objective}
\mathbb{E}_{x_0\sim \mathcal{D}}\Big[\sum_{t=0}^{\infty}(x_t^\top Qx_t+u_t^\top Ru_t) \Big],
\end{equation}
where initial state $x_0$ is assumed to be randomly distributed according to a distribution $\mathcal{D}$, and $Q\in \mathbb{R}^{n\times n}$ and $R\in \mathbb{R}^{m\times m}$ are positive definite matrices.  

The static output feedback is described by the form
\begin{equation}   
\label{eq.sop}
u_t=-Ky_t,
\end{equation} 
where $K\in \mathbb{R}^{m\times d}$. Then the closed loop system can be written as 
\begin{equation}   
\label{eq.closed-loop-system}
x_{t+1} = (A-BKC)x_t,
\end{equation}
and the objective function becomes
\begin{equation}   
\label{eq.objective_with_K)}
J(K) = \mathbb{E}_{x_0\sim \mathcal{D}}\Big[\sum_{t=0}^{\infty}x_t^\top (Q+C^\top K^\top R K C )x_t \Big].
\end{equation}
We assume that a stabilizing controller exists. We denote the feasible set which contains all stabilizing  $K$ as $\mathbb{K}$, that is 
\begin{equation}
\nonumber
\mathbb{K}=\{K\in \mathbb{R}^{m\times d}:\rho(A-BKC)<1\}.
\end{equation}

Define $P_K$ as the solution to the Lyapunov equation
\begin{equation}   
\label{eq.lyapunov_equation}
P_K = Q + C^\top K^\top RKC+(A-BKC)^\top P_K(A-BKC),
\end{equation}
and then the cost function can be expressed by
\begin{equation} 
\label{eq.cost_in_P}
J(K)= \mathbb{E}_{x_0\sim \mathcal{D}}x_0^\top P_Kx_0 = {\rm Tr}(P_KX_0),
\end{equation}
where 
\begin{equation}
\nonumber
X_0 :=  \mathbb{E}_{x_0\sim \mathcal{D}}x_0x_0^\top.
\end{equation}
Define the state correlation matrix as 
\begin{equation}
\nonumber
\Sigma_K:=\mathbb{E}_{x_0\sim \mathcal{D}}\sum_{t=0}^{\infty}x_t x_t^\top,
\end{equation}
then it is straightforward that
\begin{equation}
\nonumber
\mu :=  \sigma_{\rm min}(X_0)\
\le \sigma_{\rm min}(\Sigma_K).
\end{equation}
It is also well-known that $\Sigma_K$ satisfies the following Lyapunov equation
\begin{equation}   
\label{eq.sigma_lyapunov_equation}
\Sigma_K = X_0+(A-BKC) \Sigma_K(A-BKC)^\top.
\end{equation}

\section{Gradient and Cost Properties}
\label{sec.gradient_property}
In this section, we first derive the formula of the gradient and Hessian, and then establish several key properties of the SOF cost. 

\subsection{Gradient and Second Derivative}
\label{sec.gradient}

\begin{lemma}[Policy Gradient Expression]
\label{lemma:gradient}
 For $\forall  K \in \mathbb{K}$, the policy gradient is 
\begin{equation} 
\label{eq.gradient}
\nabla J(K) = 2E_K\Sigma_KC^\top,
\end{equation}
where $E_K = (R+B^\top P_K B)KC-B^\top P_KA$.
\end{lemma}
\begin{proof}
Define the value function of state $x_t$ as 
\begin{equation}
\nonumber
V_K(x_t): = x_t^\top P_Kx_t.
\end{equation}
From \eqref{eq.lyapunov_equation}, it follows that 
\begin{equation}
\nonumber
\begin{aligned}
V_K(x_0) =& x_0^\top P_Kx_0\\
=& x_0^\top (Q + C^\top K^\top RKC)x_0\\
&+x_0^\top(A-BKC)^\top P_K(A-BKC)x_0\\
=&x_0^\top (Q + C^\top K^\top RKC)x_0+V_K((A-BKC)x_0).
\end{aligned}
\end{equation}
Taking the gradient of $V_K(x_0)$ w.r.t. $K$, one has
\begin{equation}
\nonumber
\begin{aligned}
\nabla V_K(x_0) =& 2RKC x_0x_0^\top C^\top\\
&-2B^\top P_K(A-BKC)x_0x_0^\top C^\top\\
&+x_1^\top \nabla P_K x_1\big|_{x_1=(A-BKC)x_0}\\
=& 2[(R+B^\top P_K B)KC-B^\top P_K A] x_0x_0^\top C^\top\\
&+\nabla V_K(x_1)\big|_{x_1=(A-BKC)x_0}\\
=& 2[(R+B^\top P_K B)KC-B^\top P_K A] \sum_{t=0}^{\infty}(x_tx_t^\top) C^\top .
\end{aligned}
\end{equation}
Then, we can observe that 
\begin{equation}  
\nonumber
\nabla J(K) = \mathbb{E}_{x_0\sim \mathcal{D}}\nabla V_K(x_0)=2E_K\Sigma_KC^\top.
\end{equation}
\end{proof}

The SOF cost $J(K)$ is twice differentiable. To avoid tensors, we analyze the Hessian of the SOF cost $J(K)$ applied to a nonzero direction $Z\in \mathbb{R}^{m\times d}$, which is given as 
\begin{equation}  
\label{eq.quadaratic_of_hessian}
\begin{aligned}
\nabla^2J(K)[Z,Z]:&=\frac{d^2}{d\eta^2}\Big|_{\eta=0}J(K+\eta Z)\\
&= \mathbb{E}_{x_0\sim \mathcal{D}}x_0^\top \frac{d^2}{d\eta^2}\Big|_{\eta=0}P_{K+\eta Z}x_0\\
&={\rm Tr}(\frac{d^2}{d\eta^2}\Big|_{\eta=0}P_{K+\eta Z}X_0).
\end{aligned}
\end{equation}

\begin{lemma}
\label{lamma:hessian}
For $\forall  K \in \mathbb{K}$, the Hessian of the SOF cost $J(K)$ applied to a direction $Z\in \mathbb{R}^{m\times d}$ satisfies 
\begin{equation}  
\label{eq.Hessian formular}
\begin{aligned}
\nabla^2J(K)[Z,Z]&=2{\rm Tr}(C^\top Z^\top (R+B^TP_KB)ZC \Sigma_K)\\
&\qquad - 4{\rm Tr}((BZC)^\top P'_K[Z](A-BKC) \Sigma_K).
\end{aligned}
\end{equation}
where 
\begin{equation} 
\label{eq.derivative_P}
\begin{aligned}
P'_K&[Z]=\\
&\sum_{j=0}^{\infty}{(A-BKC)^\top}^j(C^\top Z^\top E_K  + E_K^\top ZC)(A-BKC)^j.
\end{aligned}
\end{equation}
\end{lemma}

\begin{proof}
Denote $P'_K[Z]=\frac{d}{d\eta}\Big|_{\eta=0}P_{K+\eta Z}$. From \eqref{eq.lyapunov_equation}, we have
\begin{equation} 
\label{eq.P_derivative}
\begin{aligned}
&P'_K[Z]\\
&=C^\top Z^\top R KC+C^\top K^\top R ZC \\
&\quad- (BZC)^\top P_K(A-BKC) - (A-BKC)^\top P_K BZC\\
&\quad +(A-BKC)^\top P'_K[Z](A-BKC)\\
&=C^\top Z^\top E_K  + E_K^\top ZC\\ &\quad+(A-BKC)^\top P'_K[Z](A-BKC)\\
&= \sum_{j=0}^{\infty}{(A-BKC)^\top}^j(C^\top Z^\top E_K  + E_K^\top ZC)(A-BKC)^j.
\end{aligned}
\end{equation}
Furthermore, its second derivative can be derived as
\begin{equation}   
\label{eq.second_derivative_P}
\begin{aligned}
&\frac{d^2}{d\eta^2}\Big|_{\eta=0}P_{K+\eta Z}\\
&=C^\top Z^\top \big((R+B^TP_KB)ZC-B^\top P'_K[Z](A-BKC)\big)  \\
&\quad + \big((R+B^TP_KB)ZC-B^\top P'_K[Z](A-BKC)\big)^\top ZC\\ 
&\quad-(BZC)^\top P'_K[Z](A-BKC)\\
&\quad-(A-BKC)^\top P'_K[Z]BZC\\
&\quad+(A-BKC)^\top \frac{d^2}{d\eta^2}\Big|_{\eta=0}P_{K+\eta Z}(A-BKC)\\
&=\sum_{j=0}^{\infty}{(A-BKC)^\top}^jS_1(A-BKC)^j,
\end{aligned}
\end{equation}
where
\begin{equation}
\nonumber
\begin{aligned}
S_1:=&2\big(C^\top  Z^\top (R+B^TP_KB)ZC\\
&-(BZC)^\top P'_K[Z](A-BKC)\\
&-(A-BKC)^\top P'_K[Z]BZC\big).
\end{aligned}
\end{equation}
Finally, using \eqref{eq.quadaratic_of_hessian}, we can show that 
\begin{equation}  
\nonumber
\begin{aligned}
\nabla^2J(K)[Z,Z]&={\rm Tr}(\sum_{j=0}^{\infty}{(A-BKC)^\top}^jS_1(A-BKC)^jX_0)\\
&={\rm Tr}(S_1\sum_{j=0}^{\infty}{(A-BKC)^jX_0(A-BKC)^\top}^j)\\
&={\rm Tr}(S_1\Sigma_K)\\
&=2{\rm Tr}(C^\top Z^\top (R+B^TP_KB)ZC \Sigma_K)\\
&\qquad - 4{\rm Tr}((BZC)^\top P'_K[Z](A-BKC) \Sigma_K).
\end{aligned}
\end{equation}
\end{proof}

\subsection{Cost Function Properties}
\label{sec.property}
Given the analytical policy gradient, before moving forward to analyze the performance of gradient-based methods in solving the SOF problem, we first need to establish several key properties of the SOF cost. The intermediate lemmas required by the property analysis are provided in Appendix \ref{appen.intermediate lemma}.

\begin{lemma}[Coercivity]
\label{lemma.coercivity}
The cost function $J(K)$ is coercive in the sense that for any sequence $\{K_i\}_{i=1}^{\infty}\subseteq \mathbb{K}$ 
\begin{equation}
\nonumber
J(K_i) \rightarrow +\infty,\quad {\rm if}\   \|K_i\|\rightarrow +\infty\  {\rm or}\  K_i \rightarrow K \in \partial \mathbb{K}.
\end{equation}
\end{lemma}
\begin{proof}
From \eqref{eq.cost_in_P}, we can show that
\begin{equation}   
\nonumber
\begin{aligned}
J(K_i) &= \mathbb{E}_{x_0\sim \mathcal{D}}\Big[\sum_{t=0}^{\infty}x_t^\top (Q+C^\top K_i^\top RK_iC)x_t \Big]\\
&={\rm Tr}((Q+C^\top K_i^\top RK_iC)\Sigma_{K_i})\\
&\ge\mu\sigma_{\rm min}(R)\sigma_{\rm min}(CC^\top)\| K_i\|_F^2\\
&\ge\mu\sigma_{\rm min}(R)\sigma_{\rm min}(C)^2\| K_i\|^2,
\end{aligned}
\end{equation}
which directly leads to that $J(K_i)\rightarrow +\infty$ as $\|K_i\|\rightarrow +\infty$.

From \eqref{eq.lyapunov_equation}, $P_K$ can be rewritten as
\begin{equation}   
\label{eq.P_expand}
\begin{aligned}
P_K= \sum_{j=0}^{\infty}{(A-BKC)^\top}^j(Q+C^\top K^\top RKC)(A-BKC)^j.
\end{aligned}
\end{equation}
Therefore, we have
\begin{equation}   
\nonumber
\begin{aligned}
&J(K_i) =\\
&{\rm Tr}(\sum_{j=0}^{\infty}{(A-BK_iC)^\top}^j(Q+C^\top K_i^\top RK_iC)(A-BK_iC)^jX_0)\\
&\ge \mu \sigma_{\rm min}(Q)\sum_{t=0}^{\infty}\|(A-BK_iC)^j\|_F^2\\
&\ge \mu \sigma_{\rm min}(Q)\sum_{t=0}^{\infty}\|(A-BK_iC)^j\|^2\\
&\ge \mu \sigma_{\rm min}(Q)\sum_{t=0}^{\infty}\rho(A-BK_iC)^{2j}\\
&= \mu \sigma_{\rm min}(Q)\frac{1-\rho(A-BK_iC)^{\infty}}{1-\rho(A-BK_iC)^2}.
\end{aligned}
\end{equation}
Since $ \rho(A-BKC)=1$ when $K \in \partial \mathbb{K}$, by continuity of the $\rho(A-BK_iC)$, we have $\rho(A-BK_iC) \rightarrow 1$ as $K_i \rightarrow K \in \partial \mathbb{K}$. Thereby, for every $\epsilon >0$, there exists some $N(\epsilon)\in \mathbb{N}$ such that $ 1 - \rho(A-BK_iC) < \epsilon$ for all $i \ge N(\epsilon)$. That is $1>\rho(A-BK_iC)>1-\epsilon$ for $i\ge N(\epsilon)$. Thereby, $J(K_i)$ is bounded below by
\begin{equation} 
\nonumber
J(K_i) \ge \mu \sigma_{\rm min}(Q)\frac{1}{1-(1-\epsilon)^2}.
\end{equation}
It thus follows that
$J(K_i) \rightarrow +\infty$ as $K_i \rightarrow \partial \mathbb{K}$. This completes the proof of Lemma \ref{lemma.coercivity}.
\end{proof}

With the coercive property in place, we can easily show the compactness on the sublevel set of the SOF cost.
\begin{lemma}[Compactness on the sublevel set]
\label{lemma.compact}
The sublevel set $\mathbb{K}_{\alpha}=\{K\in \mathbb{K}: J(K)\le \alpha\}$ is compact for $\forall \alpha \ge J(K^*)$. 
\end{lemma}
\begin{proof}
With the coercive property in place, by  \cite[Proposition 11.12]{bauschke2011convex}, we know that $\mathbb{K}_{\alpha}$ is bounded.  Since $J(K)$ is continuous on $\mathbb{K}$, the set $\mathbb{K}_{\alpha}$ is also closed, which completes the proof.
\end{proof}
Based on the compactness of the sublevel set, we can finally show that under certain conditions, the decrease of the cost ensures the line segment between two adjacent iterates stays inside $\mathbb{K}_{\alpha} $.

\begin{lemma}[Smoothness on the sublevel set]
\label{lemma.L-smooth}
For any $K\in \mathbb{K}_\alpha$, we have  $\|\nabla^2 J({\rm vec}(K))\|\le L$, where the smoothness constant $L$ is 
\begin{equation}
\nonumber
L=2\|C\|^2\Big(\|R\|+\|B\|^2(1+\frac{2\zeta_1}{\|C\|\|B\|})\frac{\alpha}{\mu}\Big)\frac{\alpha}{\sigma_{\rm min}(Q)},
\end{equation}
with 
\begin{equation}
\label{eq.zeta1}
\zeta_1 = \frac{1}{\sigma_{\rm min}(Q)}\Big((1+\|C\|^2\|B\|^2)\frac{\alpha}{\mu}+\|C\|^2\|R\| \Big)-1.
\end{equation}
This means that for any $K,K'\in \mathbb{K}_\alpha$ satisfying $K+\delta(K'-K)\in \mathbb{K}_{\alpha}$, $\forall \delta\in [0,1]$, the following inequality holds
\begin{equation}
\label{eq.L-smooth}
\begin{aligned}
J(K')&\le J(K)+{\rm Tr}(\nabla J(K)^\top(K'-K))+\frac{L}{2}\|K'-K\|_F^2.
\end{aligned}
\end{equation}
\end{lemma}

\begin{proof}
From \eqref{eq.quadaratic_of_hessian}, applying the Taylor series expansion about direction $Z$, we can show that 
\begin{equation}
\nonumber
\nabla^2J(K)[Z,Z]={\rm vec}(Z)^\top \nabla^2J(\rm vec(K)) {\rm vec}(Z).
\end{equation}
Since $\nabla^2 J({\rm vec}(K))$ is symmetric, one has
\begin{equation}
\label{eq.supremum_of_hessian}
\begin{aligned}
\|\nabla^2 J({\rm vec}(K))\|&=\sup_{\|Z\|_F=1}|{\rm vec}(Z)^\top \nabla^2J(\rm vec(K)) {\rm vec}(Z)|\\
&=\sup_{\|Z\|_F=1}|\nabla^2J(K)[Z,Z]|.
\end{aligned}
\end{equation}

Using~\eqref{eq.Hessian formular}, we further have 
\begin{equation}
\label{eq.upper_bound_of_hessian}
\begin{aligned}
\|\nabla^2 J&({\rm vec}(K))\|\\
\le& 2\sup_{\|Z\|_F=1}|{\rm Tr}(C^\top Z^\top (R+B^TP_KB)ZC \Sigma_K)|\\
&+4\sup_{\|Z\|_F=1}|{\rm Tr}((BZC)^\top P'_K[Z](A-BKC) \Sigma_K)|\\
=:&2q_1+4q_2.
\end{aligned}
\end{equation}
Therefore, we only need to analyze the upper bounds of $q_1$ and $q_2$, separately. Firstly, we have
\begin{equation}
\label{eq.upper_of_q1}
\begin{aligned}
 q_1 \le &\sup_{\|Z\|_F=1}(\|C^\top Z^\top (R+B^TP_KB)ZC\|{\rm Tr}(\Sigma_K))\\
 \le &\sup_{\|Z\|_F=1}(\|C\|^2\|Z\|^2\|R+B^TP_KB\|{\rm Tr}(\Sigma_K))\\
 \le &\sup_{\|Z\|_F=1}( \|C\|^2\|Z\|_F^2(\|R\|+\|B\|^2\|P_K\|){\rm Tr}(\Sigma_K))\\
  \le & \|C\|^2(\|R\|+\|B\|^2\frac{J(K)}{\mu})\frac{J(K)}{\sigma_{\rm min}(Q)},
 \end{aligned}
\end{equation}
where the last step follows from Lemma \ref{lemma.upper_bound} in the appendix.

Next, by the Cauchy-Schwarz inequality, we have
\begin{equation}
\label{eq.upper_of_q2}
\begin{aligned}
q_2
\le &\sup_{\|Z\|_F=1}(\|(BZC)^\top P'_K[Z](A-BKC)\Sigma_K^{1/2}\|_F \|\Sigma_K^{1/2}\|_F)\\
\le& \sup_{\|Z\|_F=1}\big(\|C\|\|Z\|\|B\|\|P'_K[Z]\|\\
&\qquad\qquad\qquad
\|(A-BKC)\Sigma_K^{1/2}\|_F \sqrt{{\rm Tr}(\Sigma_K)}\big)\\
\le& \|C\|\|B\|\sup_{\|Z\|_F=1}(\|P'_K[Z]\|)\\
&\qquad \sqrt{{\rm Tr}((A-BKC)\Sigma_K(A-BKC)^\top)} \sqrt{{\rm Tr}(\Sigma_K)}\\
\le & \|C\|\|B\|{\rm Tr}(\Sigma_K)\sup_{\|Z\|_F=1}\|P'_K[Z]\|\\
\le & \|C\|\|B\|\frac{J(K)}{\sigma_{\rm min}(Q)}\sup_{\|Z\|_F=1}\|P'_K[Z]\|,
\end{aligned}
\end{equation}
where the last step follows from Lemma \ref{lemma.upper_bound}.
Combining \eqref{eq.upper_of_q1} and \eqref{eq.upper_of_q2}, the only thing left is to show the following bound holds:
\begin{equation}
\nonumber
\sup_{\|Z\|_F=1}\|P'_K[Z]\| \le \zeta_1 \|P_K\|,
\end{equation}
where $\zeta_1$ is as given by \eqref{eq.zeta1}.

We will prove the above inequality by showing that $P'_K[Z] \preceq \zeta_1 P_K$. Given \eqref{eq.lyapunov_equation} and \eqref{eq.derivative_P}, if $(C^\top Z^\top E_K  + E_K^\top ZC)\preceq \zeta_1(Q+C^\top K^\top RKC)$ for some $\zeta_1 \in \mathbb{R}^+$, we immediately have $P'_K[Z] \preceq \zeta_1 P_K$. Therefore, it remains to find $\zeta_1$. It follows from \eqref{eq.lyapunov_equation} that
\begin{equation}
\begin{aligned}
&C^\top Z^\top E_K  + E_K^\top ZC \\
&=C^TZ^TRKC+C^TK^TRZC-C^TZ^TB^TP_K(A-BKC)\\
&\quad -(A-BKC)^TP_KBZC\\
&\preceq  C^TZ^TRZC+C^TK^TRKC\\
&\qquad+(A-BKC)^\top P_K(A-BKC) +(BZC)^TP_KBZC\\
&= P_K - Q + C^TZ^TRZC + (BZC)^TP_KBZC\\
&\preceq  \|P_K+C^TZ^TRZC + (BZC)^TP_KBZC\|I-Q\\
&\preceq  (\|P_K\|+\|C\|^2\|R\| + \|C\|^2\|B\|^2\|P_K\|)I-Q\\
&\preceq  \Big((1+\|C\|^2\|B\|^2)\frac{\alpha}{\mu}+\|C\|^2\|R\| \Big)\frac{Q}{\sigma_{\rm min}(Q)}-Q.
\end{aligned}
\end{equation}
Therefore, we proved that $P'_K[Z] \preceq \zeta_1 P_K$, which directly leads to
\begin{equation}
\label{eq.upper_of_q2_new}
\begin{aligned}
q_2\le &\zeta_1\|C\|\|B\|\|P_K\|\frac{J(K)}{\sigma_{\rm min}(Q)}\\
\le & \zeta_1\|C\|\|B\|\frac{J(K)^2}{\mu\sigma_{\rm min}(Q)}.
\end{aligned}
\end{equation}
Combining \eqref{eq.upper_bound_of_hessian}, \eqref{eq.upper_of_q1}, and \eqref{eq.upper_of_q2_new}, for $K\in \mathbb{K}_\alpha$, we finally have 
\begin{equation}
\nonumber
\begin{aligned}
&\|\nabla^2 J({\rm vec}(K))\|\\
&\le2\|C\|^2(\|R\|+\|B\|^2\frac{\alpha}{\mu})\frac{\alpha}{\sigma_{\rm min}(Q)}+4\zeta_1\|C\|\|B\|\frac{\alpha^2}{\mu\sigma_{\rm min}(Q)}\\
&\le  2\|C\|^2\Big(\|R\|+\|B\|^2(1+\frac{2\zeta_1}{\|C\|\|B\|})\frac{\alpha}{\mu}\Big)\frac{\alpha}{\sigma_{\rm min}(Q)}.
\end{aligned}
\end{equation}
\end{proof}

If $J(K)$ is globally $L$-smooth, it is well known that gradient descent can converge to a stationary point with a dimension-free iteration complexity \cite{polyak1963gradient,nesterov1998introductory}. However, Lemma \ref{lemma.L-smooth} only shows that  $\|\nabla^2 J({\rm vec}(K))\|\le L$ for $K\in\mathbb{K}_{\alpha}$. Since the sublevel set $\mathbb{K}_{\alpha}$ can be nonconvex with a non-smooth boundary, or even disconnected \cite{fatkhullin2020CTSOF}, the inequality~\eqref{eq.L-smooth} cannot be directly applied to any $K, K'\in \mathbb{K}_\alpha$. Hence, the convergence analysis is non-trivial.

\begin{lemma}[Performance difference lemma]
\label{lemma.performance_diff}
Suppose $K, K' \in \mathbb{K}$. It holds that:
\begin{equation}
\nonumber
\begin{aligned}
J&(K')-J(K) = 2{\rm Tr}\big(\Sigma_{K'}(K'C-KC)^\top E_K\big)+\\
&+{\rm Tr}\big(\Sigma_{K'}(K'C-KC)^\top (R+B^\top P_K B)(K'C-KC)\big).
\end{aligned}
\end{equation}
\end{lemma}
We provide a brief proof of Lemma \ref{lemma.performance_diff} in Appendix \ref{append.proof}. The performance difference lemma, also known as almost smoothness \cite{fazel2018global}, is the basis for deriving the following gradient domination condition.

\begin{lemma}[Gradient domination]
\label{lemma.gradient_dominance}
Define the set $\mathbb{C}$ as $\mathbb{C}=\{C \in \mathbb{R}^{n\times n}: {\rm rank}(C)=n\}$. Let $K^*$ be the global optimal SOF policy, and $J_s^*$ be the global optimal cost of the corresponding state feedback LQR. Suppose $K$ has finite cost and $\mu > 0$. It holds that \begin{equation}
\label{eq.gradient_dominance_1}
J(K)-J(K^*)\le \frac{\|\Sigma_{K^*}\|{\rm Tr}\big(E_K^\top E_K\big)}{\sigma_{\rm min}(R)}.
\end{equation}
If $C\in \mathbb{C}$, one further has
\begin{equation}
\label{eq.gradient_dominance}
J(K)-J_s^*\le \frac{\|\Sigma_{K^*}\|\|\nabla J(K)\|_F^2}{4\mu^2\sigma_{\rm min}(C)^2\sigma_{\rm min}(R)}, \ \forall C\in \mathbb{C}.
\end{equation}
Moreover, we have the following lower bound 
\begin{equation}
\label{eq:lower_bound_2}
J(K)-J_s^*\ge \frac{\mu{\rm Tr}\big(E_K^\top  E_K\big)}{\|R+B^\top P_K B\|}, \ \forall C \in \mathbb{C}.
\end{equation}
\end{lemma}

These results can be proved in a similar way to \cite[Lemma 3]{fazel2018global} where $C=I_n$ (see Appendix \ref{appendix.dominance} for details).
The gradient dominance condition indicates that there are no stationary points other than the global minimum, which is crucial to achieve global convergence for policy gradient \cite{polyak1963gradient,lojasiewicz1963topological}. However, when ${\rm rank}(C)<n$, gradient dominance does not hold. Therefore, in the latter case, it is hard to expect anything better than converging to a stationary point.

\begin{lemma}[$M$-Lipschitz continuous Hessian]
\label{lemma.continue_Hessian}
 For any $K$, $K'\in \mathbb{K}_\alpha$ satisfying $K+\delta(K'-K)\in \mathbb{K}_{\alpha}$, $\forall \delta\in [0,1]$, the following inequality holds
\begin{equation}
\nonumber
\begin{aligned}
\|\nabla^2 J(\mathrm{vec}(K))-\nabla^2 J(\mathrm{vec}(K'))\|\le M \|K-K'\|,
\end{aligned}
\end{equation}
where 
\begin{equation}
\nonumber
M =     \frac{2\alpha^2}{\mu\sigma_{\rm min}(Q)}\|B\|\|C\|\big((2\zeta_1+\zeta_2)\|B\|\|C\|+2\zeta_3+\zeta_4\big),
\end{equation}
and $\zeta_{\{1,2,3,4\}}$ are defined in Lemma \ref{lemma.L-smooth} and \ref{lemma.positive_definite}.
\end{lemma}
\begin{proof}
Similar to \eqref{eq.supremum_of_hessian}, since $\nabla^2 J(\mathrm{vec}(K))-\nabla^2 J(\mathrm{vec}(K'))$ is symmetric, we have
\begin{equation}
\nonumber
\begin{aligned}
&\|\nabla^2 J(\mathrm{vec}(K))-\nabla^2 J(\mathrm{vec}(K'))\|\\
&=\sup_{\|Z\|_F=1}|\nabla^2J(K)[Z,Z]-\nabla^2J(K')[Z,Z]|.
\end{aligned}
\end{equation}
By \eqref{eq.quadaratic_of_hessian}, we define
\begin{equation}
\nonumber
g(\delta):=\nabla^2 J(K+\delta(K'-K))[Z,Z],
\end{equation}
and denote 
$\Bar{K}:=K+\delta(K'-K)$, $\Delta K:= K'-K$, and $P''_K[Z]:=\frac{d^2}{d\eta^2}\Big|_{\eta=0}P_{K+\eta Z}$. Then, from \eqref{eq.second_derivative_P}, one has
\begin{equation}  
\nonumber
\begin{aligned}
g'(\delta)&= \mathbb{E}_{x_0\sim \mathcal{D}}\left(x_0^\top \frac{\partial^3}{\partial\eta^2\partial \delta}\Big|_{\eta=0}P_{K+\delta(K'-K)+\eta Z}x_0\right).
\end{aligned}
\end{equation}
By the fundamental theorem of calculus, it follows that
\begin{equation}
\label{eq.hessian_diff_calculus}
\begin{aligned}
\|\nabla^2 J(\mathrm{vec}(K))-\nabla^2& J(\mathrm{vec}(K'))\|\\
&=\sup_{\|Z\|_F=1}|g(0)-g(1)|\\
&=\sup_{\|Z\|_F=1}|\int_{0}^{1}g'(\delta){\rm d}\delta|\\
&\le\int_{0}^{1}\sup_{\|Z\|_F=1}|g'(\delta)|{\rm d}\delta.
\end{aligned}
\end{equation}
Based on \eqref{eq.second_derivative_P}, we can observe that 
\begin{equation}  
\nonumber
\begin{aligned}
&\frac{\partial^3}{\partial\eta^2\partial \delta}\Big|_{\eta=0}P_{K+\delta(K'-K)+\eta Z}\\
&\qquad = \sum_{j=0}^{\infty}{(A-B\bar{K}C)^\top}^jS_2(A-B\bar{K}C)^j,
\end{aligned}
\end{equation}
where 
\begin{equation}   
\nonumber
\begin{aligned}
&S_2:=2C^\top Z^\top B^T\frac{\partial P_{\bar{K}}}{\partial \delta}BZC + 2C^\top Z^\top B^T P'_{\bar{K}}[Z] B\Delta K C\\
& \qquad + 2C^\top \Delta K^\top B^T P'_{\bar{K}}[Z] BZ C\\
&\qquad-2(BZC)^\top\frac{\partial P'_{\bar{K}}[Z]}{\partial \delta}(A-B\bar{K}C) \\
&\qquad-2(A-B\bar{K}C)^\top\frac{\partial P'_{\bar{K}}[Z]}{\partial \delta}(BZC)\\
&\qquad -  (B\Delta KC)^\top P''_{\bar{K}}[Z](A-B\bar{K}C)\\
& \qquad -(A-B\bar{K}C)^\top P''_{\bar{K}}[Z](B\Delta K C).
\end{aligned}
\end{equation}
Then, it follows that 
\begin{equation}  
\begin{aligned}
g'(\delta)&={\rm Tr}(\frac{\partial^3}{\partial\eta^2\partial \delta}\Big|_{\eta=0}P_{K+\delta(K'-K)+\eta Z}X_0)\\
&=2{\rm Tr}(C^\top Z^\top B^T\frac{\partial P_{\bar{K}}}{\partial \delta}BZC \Sigma_{\bar {K}})\\
&\qquad +4{\rm Tr}(C^\top Z^\top B^T P'_{\bar{K}}[Z] B\Delta K C\Sigma_{\bar {K}})\\
&\qquad -4{\rm Tr}((BZC)^\top\frac{\partial P'_{\bar{K}}[Z]}{\partial \delta}(A-B\bar{K}C)\Sigma_{\bar {K}})\\
&\qquad - 2{\rm Tr}((B\Delta KC)^\top P''_{\bar{K}}[Z](A-B\bar{K}C)\Sigma_{\bar {K}}).
\end{aligned}
\end{equation}
Utilizing the results of Lemma \ref{lemma.positive_definite}, we can further show that 
\begin{equation}  
\label{eq.third_derivative}
\begin{aligned}
&\sup_{\|Z\|_F=1}|g'(\delta)|\\
&\le2\zeta_2\|B\|^2\|Z\|^2\|C\|^2\|P_{\bar{K}}\|{\rm Tr}(\Sigma_{\bar {K}})\|\Delta K\|\\
&\qquad +4\zeta_1\|B\|^2\|Z\|\|C\|^2\|P_{\bar{K}}\|{\rm Tr}(\Sigma_{\bar {K}})\|\Delta K\|\\
&\qquad +4\zeta_3\|B\|\|Z\|\|C\|\|A-B\bar{K}C\|\|P_{\bar{K}}\|{\rm Tr}(\Sigma_{\bar {K}})\|\Delta K\|\\
&\qquad +2\zeta_4\|B\|\|C\|\|A-B\bar{K}C\|\|P_{\bar{K}}\|{\rm Tr}(\Sigma_{\bar {K}})\|\Delta K\|\\
&\le2\|B\|\|C\|\big((2\zeta_1+\zeta_2)\|B\|\|C\|+2\zeta_3+\zeta_4\big)\frac{\alpha^2\|\Delta K\|}{\mu\sigma_{\rm min}(Q)}.
\end{aligned}
\end{equation}
By plugging \eqref{eq.third_derivative} into \eqref{eq.hessian_diff_calculus}, we finally complete the proof.
\end{proof}

The Lipschitz continuity of the Hessian can be used to obtain a stronger result of convergence to a local minimum for nonconvex optimization under mild assumptions \cite{nesterov1998introductory}. Besides, this property has also been widely used in \cite{jin2017escape,ge2015escaping,carmon2018accelerated} to achieve efficient strict saddle point escape.

\section{Convergence Analysis}
\label{sec:convergence}
In this section, we prove several novel convergence results of the gradient descent method for SOF LQR. Since $J(K)$ is nonconvex, the properties of the cost function obtained in Section \ref{sec.property} will play a key role in the convergence analysis. 

\subsection{Convergence to Stationary Points}

We consider the following gradient descent update rule
\begin{equation}
\label{eq.pg}
K_{i+1}=K_i-\eta \nabla J(K_i),
\end{equation}
where $K_0 \in \mathbb{K}$. By Lemma \ref{lemma.L-smooth}, if we can ensure that the line segment between $K_i$ and $K_i-\eta \nabla J(K_i)$ is in $\mathbb{K}_{\alpha}$,  the convergence of policy descent \eqref{eq.pg} for SOF can be proved by invoking a similar analysis to the case of global $L$-smoothness.
Before presenting the main theorem, we introduce the following definition for convenience. 
\begin{definition}
For a differentiable function $J(\cdot)$, we say $K$ is a first-order stationary point if $\|\nabla J(K)\|_F=0$; we also say $K$ is an $\epsilon$-first-order stationary point if $\|\nabla J(K)\|_F\le \epsilon$.
\end{definition}

\begin{theorem}
\label{theorem: first_order_convergence}
Suppose $J(K_0)=\alpha$, $\mu>0$, and $L$ is as defined in Lemma \ref{lemma.L-smooth}. If we run gradient descent \eqref{eq.pg} with any step size $\eta \le 1/L$, $J(K_i)$ is monotone decreasing, and the algorithm will output an $\epsilon$-first-order stationary point within the following number of iterations 
\begin{equation}
\nonumber
\frac{2\alpha}{\eta\epsilon^2}.
\end{equation}
Besides, the line segment between $K_i$ and $K_{i+1}$ is in $\mathbb{K}_\alpha$ for any $i\in \mathbb{N}$, denoted as $[K_i,K_{i+1}]\in \mathbb{K}_\alpha$. Furthermore, if $C\in \mathbb{C}$, for $\epsilon_J>0$, the iterate $K_N$ satisfies
$J(K_N)-J_s^* \le \epsilon_J$ given 
\begin{equation}
N\ge \frac{\|\Sigma_{K^*}\|}{2\eta\mu^2\sigma_{\rm min}(C)^2\sigma_{\rm min}(R)}\log (\frac{J(K_0)-J_s^*}{\epsilon_J}).
\end{equation}
\end{theorem}
\begin{proof}
Firstly, we define $\mathbb{K}_{\alpha}^o:=\{K\in \mathbb{K}: J(K)< \alpha\}$, and denote its complement as $(\mathbb{K}_{\alpha}^o)^c$, which is closed. From Lemma \ref{lemma.L-smooth}, given $\phi\in(0,1)$, there exists $\varsigma>0$ such that $\|\nabla^2J({\rm vec}(K))\|\le (1+\phi)L$ for $K \in \mathbb{K}_{\alpha+\varsigma}$. It is clear that $\mathbb{K}_{\alpha}\cap(\mathbb{K}_{\alpha+\varsigma}^o)^c=\emptyset$. 

Since $\mathbb{K}_{\alpha}$ is compact by Lemma \ref{lemma.compact}, the distance between $\mathbb{K}_{\alpha}$ and $(\mathbb{K}_{\alpha+\varsigma}^o)^c$, defined as $\omega:=\inf\{\|x-y\|, \forall x\in \mathbb{K}_{\alpha}, y\in (\mathbb{K}_{\alpha+\varsigma}^o)^c\}$, must be strictly positive. Let $\tau \in \mathbb{R}^+ \le \min \{\omega/\|\nabla J(K)\|,2/((1+\phi)L)\}$. It follows from $\tau\leq \omega/\|\nabla J(K)\|$ that the line segment between $K$ and $K-\tau \nabla J(K)$ is in $\mathbb{K}_{\alpha+\varsigma}$. Since $\|\nabla^2J({\rm vec}(K))\|\le (1+\phi)L$, by \eqref{eq.L-smooth}, we can show that
\begin{equation}
\nonumber
\begin{aligned}
J(K-\tau \nabla J(K))&\le J(K)-\tau(1-\frac{(1+\phi)L}{2}\tau)\|\nabla J(K)\|_F^2.
\end{aligned}
\end{equation}
Since $\tau \le 2/((1+\phi)L)$, we have $J(K-\tau \nabla J(K)) \le J(K)$. Thereby, $K-\tau \nabla J(K) \in \mathbb{K}_\alpha$. It is then obvious that the line segment $[K, K-\tau \nabla J(K)]\in\mathbb{K}_{\alpha}$. Then, it follows that segment $[K, K-2\tau \nabla J(K)] \in \mathbb{K}_{\alpha+\varsigma}$. We can apply the same argument based on \eqref{eq.L-smooth}  of Lemma~\ref{lemma.L-smooth} to show that $[K,K-2\tau \nabla J(K)]\in\mathbb{K}_\alpha$ if $2\tau \le 2/((1+\phi)L)$. 

By induction, we can show that the line segment $[K,K-T\tau \nabla J(K)]\in\mathbb{K}_\alpha$ if $T\tau \le 2/((1+\phi)L)$. Let $\eta \le 1/L$, and we can choose $\tau>0$ and $T \in \mathbb{N}^+$ such that $\eta = T\tau \le 2/((1+\phi)L)$. Therefore, $[K,K-\eta \nabla J(K)]\in\mathbb{K}_\alpha$. Applying \eqref{eq.L-smooth} again, we have 
\begin{equation}
\label{eq:monotony_gd}
\begin{aligned}
J(K-\eta \nabla J(K))&\le J(K)-\eta(1-\frac{L}{2}\eta)\|\nabla J(K)\|_F^2\\
&\le J(K)-\frac{\eta}{2}\|\nabla J(K)\|_F^2,
\end{aligned}
\end{equation}
where the last step utilizes the fact that $\eta \le 1/L$. Note that the boundary $1/L$ is chosen because it leads to the fastest convergence rate.

Since $J(K_0)=\alpha$, according to \eqref{eq:monotony_gd}, $K_1$ generated by \eqref{eq.pg} is in $\mathbb{K}_\alpha$. By induction, we can show that
\begin{equation}
\label{eq:monotony}
J(K_{i+1})\le J(K_i)-\frac{\eta}{2}\|\nabla J(K_i)\|_F^2
\end{equation}
holds for any $i$, which also means that $[K_i, K_{i+1}] \in \mathbb{K}_\alpha$ for any $i$. Let us sum \eqref{eq:monotony} for $i=0,\cdots, N$. We can obtain
\begin{equation}
\nonumber
\frac{\eta}{2}\sum_{i=0}^{N}\|\nabla J(K_i)\|_F^2 \le J(K_0) - J(K_{N+1})\le J(K_0) - J(K^*).
\end{equation}
Thereby, it is clear that 
\begin{equation}
\nonumber
\min_{0\le i\le N} \|\nabla J(K_i)\|_F^2 \le  \frac{2(J(K_0) - J(K^*))}{\eta N} \le \frac{2\alpha}{\eta N}.
\end{equation}
Therefore, the policy descent method will output an $\epsilon$-first-order stationary point within the following number of iterations:
\begin{equation}
\nonumber
N\le \frac{2\alpha}{\eta\epsilon^2}.
\end{equation} 

If $C\in\mathbb{C}$, by \eqref{eq:monotony} and \eqref{eq.gradient_dominance}, we further have
\begin{equation}
\nonumber
J(K_{i+1})-J(K_{i}) \le -\frac{2\eta\mu^2\sigma_{\rm min}(C)^2\sigma_{\rm min}(R)}{\|\Sigma_{K^*}\|}(J(K_i)-J_s^*),
\end{equation}
which directly leads to 
\begin{equation}
\nonumber
J(K_{i})-J_s^* \le \big(1-\frac{2\eta\mu^2\sigma_{\rm min}(C)^2\sigma_{\rm min}(R)}{\|\Sigma_{K^*}\|}\big)^i(J(K_0)-J_s^*),
\end{equation}
Therefore, when $C\in\mathbb{C}$, for 
\begin{equation}
\nonumber
N\ge \frac{\|\Sigma_{K^*}\|}{2\eta\mu^2\sigma_{\rm min}(C)^2\sigma_{\rm min}(R)}\log (\frac{J(K_0)-J_s^*}{\epsilon_J}),
\end{equation}
the gradient descent \eqref{eq.pg} enjoys that
$J(K_N)-J_s^* \le \epsilon_J$.
\end{proof}

From Theorem \ref{theorem: first_order_convergence}, given an  initial stabilizing controller, one can ensure that the gradient descent method will stay in the feasible set and lead to monotone policy improvement. In particular, if $C\in \mathbb{C}$, it globally converges to the unique minimum point at a linear rate; if ${\rm rank}(C)<n$, it converges to a stationary point. Our analysis assumes the model parameters ($A$, $B$, $C$, $Q$, $R$) to have analytical characterizations of the convergence rate, which provides a baseline to the model-free settings. In the latter case, existing policy learning techniques, such as the zeroth-order optimization approach, provide effective ways to obtain an unbiased estimation of $\nabla J(K)$ from data \cite{conn2009introduction,nesterov2017random}. The work in \cite{fazel2018global} showed that with enough samples and roll-out length, the zeroth-order optimization approach could estimate the gradient within a desired accuracy. The same can be said of discrete-time SOF as well. 

\subsection{Convergence to Local Minima}

\begin{theorem}
\label{theorem: linear_local_minima}
Let $\omega$ denote the distance between $\mathbb{K}_{\beta}$ and $(\mathbb{K}_{\alpha}^o)^c$. Suppose $\beta < \alpha$, i.e., $\mathbb{K}_\beta \subset \mathbb{K}_\alpha$. Set $L$ and $M$ as described in Lemma \ref{lemma.L-smooth} and \ref{lemma.continue_Hessian}, respectively. Assume the initial point $K_0$ be close enough to a local minimum $K^\ddag \in \mathbb{K}_{\beta}$ with positive definite Hessian, i.e., $l= \lambda_{\rm min}(\nabla^2 J({\rm vec}(K^\ddag))>0$. Define $\bar{r}:=2l/M$. If $r_0=\|K_0-K^\ddag\|_F < \bar{r}$ and
$\bar{r}r_0/(\bar{r}-r_0)\le \omega$, then the gradient descent with a fixed step-size $\eta \le 1/L$ converges with the following
rate: 
\begin{equation}
\nonumber
\|K_i-K^{\ddag}\|_F\le \frac{\bar{r}r_0}{\bar{r}-r_0}\big(\frac{1}{1+\eta l}\big)^i.
\end{equation}
\end{theorem}
\begin{proof}
Given a certain local minimum $K^{\ddag}$, define set $\mathbb{K}^{\ddag}=\{K\in \mathbb{R}^{m\times d}:\|K-K^{\ddag}\|_F\le \bar{r}r_0/(\bar{r}-r_0)\}$. Since both $\mathbb{K}_{\beta}$ and $(\mathbb{K}_{\alpha}^o)^c$ are compact,  suppose $\bar{r}r_0/(\bar{r}-r_0)\le \omega$, it is clear that $\mathbb{K}^{\ddag} \subset \mathbb{K}_{\alpha}$. For any $K,K'\in \mathbb{K}^{\ddag}$, we can show that $K+\delta(K'-K)\in \mathbb{K}^{\ddag} \subset \mathbb{K}_{\alpha}$, $\forall \delta\in [0,1]$. Then, from Lemma \ref{lemma.L-smooth} and \ref{lemma.continue_Hessian}, it is straightforward that for any $K,K'\in \mathbb{K}^{\ddag}$, it holds:
\begin{equation}
\nonumber
\begin{aligned}
\|\nabla J(K')-\nabla J(K')\|\le L \|K'-K\|,
\end{aligned}
\end{equation}
and 
\begin{equation}
\nonumber
\begin{aligned}
\|\nabla^2 J({\rm vec}(K'))-\nabla^2 J({\rm vec}(K))\|\le M \|K'-K\|.
\end{aligned}
\end{equation}
The end of the proof can be completed by directly applying \cite[Theorem 1.2.4]{nesterov1998introductory} because we are able to use $L$-smoothness and $M$-Hessian Lipschitz along the entire descent domain.
\end{proof}

Theorem \ref{theorem: linear_local_minima} implies that gradient descent will converge to the local minimum at a linear rate if the starting point is sufficiently close to one. Note that the above is linear convergence in terms of the sequence and not the function value.

\subsection{Influence of the Initial Distribution}

\begin{proposition}
\label{proposition:distribution}
Denote the stationary points of SOF as $K^{\dagger}$. If $1\le{\rm rank}(C)<n$, $K^{\dagger}$ depends on the initial state distribution $\mathcal{D}$, unless $K^{\dagger}C= K_s^*$, where $K_s^*$ is the optimal control gain of the state feedback LQR (i.e., $C=I_n$).
\end{proposition}
\begin{proof}
By \eqref{eq.gradient} in \Cref{lemma:gradient} and \eqref{eq.gradient_dominance} in Lemma \ref{lemma.gradient_dominance}, we have 
\begin{equation}
\label{eq.E_K=0}
\|E_{K^*}\|_F=0, \;\text{for}\; C \in \mathbb{C}.
\end{equation}
From \eqref{eq.sop}, one has $$u_t=-Ky_t=-KCx_t,$$
which means $KC$ can be deemed as a state-feedback gain.

By the gradient dominance property, we know the stationary point $K_s^*$ of state feedback LQR (when $C=I_n$) is unique, it is straightforward that $$\|E_{K}\|_F=0 \;\text{iff}\; KC=K_s^*.$$
Note that if $\text{rank}(C)<n$, such an output feedback $K$ that makes $KC=K_s^*$ hold may not exist.

The classical control theory \cite[Lemma 2]{lee2018primal} shows that \eqref{eq.sigma_lyapunov_equation} has a unique positive definite solution $\Sigma_{K}$ if $K\in \mathbb{K}$ and $X_0\succ 0$. Given two different initial distributions $\mathcal{D}'\neq \mathcal{D}$, it follows from \eqref{eq.sigma_lyapunov_equation} that 
$$\Sigma_{K^{\dagger}}'\neq \Sigma_{K^{\dagger}} \;{\rm if}\; X_0= \mathbb{E}_{x_0\sim \mathcal{D}}x_0x_0^\top \neq X_0'= \mathbb{E}_{x_0\sim \mathcal{D}'}x_0x_0^\top.$$

From \eqref{eq.gradient}, given a stationary point $K^{\dagger}$, we can observe that 
$$\|\nabla J(K^{\dagger})\|_F = 2\|E_{K^{\dagger}}\Sigma_{K^{\dagger}}C^\top\|_F=0.$$ 

However, if $\|E_{K^{\dagger}}\|_F\neq 0$ (i.e., $K^{\dagger}C\neq K_s^*$), we cannot ensure that $\|\nabla J(K^{\dagger})\|_F=\|E_{K^{\dagger}}\Sigma_{K^{\dagger}}'C^\top\|_F=0$ for all possible distribution $\mathcal{D}'$ due to the impact of $\mathcal{D}'$ on $\Sigma_{K^{\dagger}}'$. This indicates that the stationary point of SOF when $C \notin \mathbb{C}$ depends on the initial state distribution $\mathcal{D}$, that is, different initial distribution will lead to a different stationary point, unless $K^{\dagger}C= K_s^*$. 
\end{proof}
Proposition \ref{proposition:distribution} suggests that in practice, the initial state distribution should be chosen to be close to the actual application scenario to minimize the expected cost.

\section{Conclusion}
\label{sec:conclusion}
In this paper, we have analyzed the optimization landscape of the gradient descent method for the static output feedback (SOF)  linear quadratic control. First, we established several important properties of the SOF cost function, such as coercivity, $L$-smoothness, and $M$-Lipschitz  continuous  Hessian. Based on these results, we showed that the gradient descent method converges to the global optimal controller at a linear rate in the fully observed case. Novel results on convergence (and rate of convergence) to stationary points were obtained in the partially observed case. The dependence of the stationary points on the initial state distribution is characterized explicitly. We further demonstrated that under mild assumptions, gradient descent  converges linearly to a local minimum if the starting point is  close to one for partially observed SOF. Our results provided a baseline for evaluating the sample complexities of other model-free gradient descent methods in reinforcement learning, where only estimated gradients are utilized.

\section*{Acknowledgment}
We would like to acknowledge Yang Zheng and Shengbo Eben Li for their valuable suggestions.

\appendix

\subsection{Intermediate Lemmas}
\label{appen.intermediate lemma}
\begin{lemma}
\label{lemma.upper_bound}
The upper bound of $\|P_K\|$ and $\|\Sigma_K\|$ are given by 
\begin{equation}
\nonumber
\|P_K\|\le \frac{J(K)}{\mu}, \qquad\|\Sigma_K\|\le {\rm Tr}(\Sigma_K)\le \frac{J(K)}{\sigma_{\rm min}(Q)}.
\end{equation}
\end{lemma}
\begin{proof}
From \eqref{eq.cost_in_P}, one has
\begin{equation}
\nonumber
J(K)= {\rm Tr}(P_KX_0)\ge \mu\|P_K\|.
\end{equation}
Then, the first claim can be directly obtained. Similarly, $J(K)$ can also be lower bounded by
\begin{equation}
\nonumber
\begin{aligned}
J(K)= &{\rm Tr}((Q+C^\top K^\top R K C )\Sigma_K)\\
\ge& \sigma_{\rm min}(Q){\rm Tr}(\Sigma_K)\\
\ge & \sigma_{\rm min}(Q)\|\Sigma_K\|,
\end{aligned}
\end{equation}
which leads to the second claim.
\end{proof}

\begin{lemma}
\label{lemma.K_upper_bound}
For any $K\in \mathbb{K}_\alpha$, it holds that 
\begin{equation}
\nonumber
\|KC\|\le q_3:=\frac{\sqrt{\|R\|+\|B\|^2 \alpha^2/\mu }}{\sqrt{\mu}\sigma_{\rm min}(R)}+\frac{\|B\| \|A\|\alpha}{\mu\sigma_{\rm min}(R)}.
\end{equation}
\end{lemma}

\begin{proof}
First, we can observe that
\begin{equation}
\nonumber
\begin{aligned}
&\|KC\|\\
&= \|(R+B^\top P_K B)^{-1}(R+B^\top P_K B)KC \|\\
&\le \|(R+B^\top P_K B)^{-1}\|\|(R+B^\top P_K B)KC \|\\
&\le \frac{1}{\sigma_{\rm min}(R)}\|(R+B^\top P_K B)KC-B^\top P_KA+B^\top P_KA \|\\
& \le \frac{\|E_K\|+\|B^\top P_KA\|}{\sigma_{\rm min}(R)}\\
& \le \frac{\sqrt{{\rm Tr}(E_K^\top E_K)}+\|B^\top P_KA\|}{\sigma_{\rm min}(R)}.
\end{aligned}
\end{equation}

From \eqref{eq:lower_bound_2}, we know that 
\begin{equation}
\nonumber
\begin{aligned}
{\rm Tr}(E_K^\top E_K) &\le \frac{\|R+B^\top P_K B\|}{\mu}(J(K)-J_S^*)\\
&\le \frac{\|R+B^\top P_K B\|}{\mu}J(K).
\end{aligned}
\end{equation}
Thereby, we finally have
\begin{equation}
\nonumber
\begin{aligned}
\|KC\| &\le \frac{\sqrt{\|R+B^\top P_K B\|J(K)}}{\sqrt{\mu}\sigma_{\rm min}(R)}+\frac{\|B^\top P_KA\|}{\sigma_{\rm min}(R)}\\
&\le\frac{\sqrt{\|R\|J(K)+\|B\|^2 \|J(K)\|^2/\mu }}{\sqrt{\mu}\sigma_{\rm min}(R)}+\frac{\|B\| \|A\|J(K)}{\mu\sigma_{\rm min}(R)}\\
&\le\frac{\sqrt{\|R\|\alpha+\|B\|^2 \alpha^2/\mu }}{\sqrt{\mu}\sigma_{\rm min}(R)}+\frac{\|B\| \|A\|\alpha}{\mu\sigma_{\rm min}(R)}.
\end{aligned}
\end{equation}
\end{proof}

\begin{lemma}
\label{lemma.positive_definite}
Define $\gamma:=\max_{K\in \mathbb{K}_{\alpha}} \|A-BKC\|$. Let
\begin{equation}
\nonumber
g(\delta):=\nabla^2 J(K+\delta(K'-K))[Z,Z].
\end{equation}
Denote 
$\Bar{K}:=K+\delta(K'-K)$, $\Delta K:= K'-K$, and $P''_K[Z]:=\frac{d^2}{d\eta^2}\Big|_{\eta=0}P_{K+\eta Z}$. Suppose $K, \Bar{K}\in \mathbb{K}_{\alpha}$, it holds that: $\frac{\partial P_{\bar{K}}}{\partial \delta} \preceq \zeta_2 \|\Delta K\|P_{\bar{K}}$, $\frac{\partial P'_{\bar{K}}[Z]}{\partial \delta} \preceq \zeta_3 \|\Delta K\|P_{\bar{K}}$,  and $P''_{\bar{K}}[Z]\preceq \zeta_4 P_{\bar{K}}$.
The expression of $\zeta_2$, $\zeta_3$, and $\zeta_4$ are expressed as 
\begin{equation}
\nonumber
\zeta_2 = \frac{2}{\sigma_{\rm min}(Q)}\left(\|C\|\|R\|q_3+\gamma\|C\|\|B\|\frac{\alpha}{\mu}\right),
\end{equation}
and
\begin{equation}
\nonumber
\begin{aligned}
\zeta_3 = \frac{2}{\sigma_{\rm min}(Q)}& (\|C\|^2\|R\|+\|C\|\|B\|(\|C\|\|B\|\\
&\qquad\qquad\qquad\qquad+\zeta_1\gamma+\zeta_2\gamma)\frac{\alpha}{\mu}),
\end{aligned}
\end{equation}
and
\begin{equation}
\nonumber
\zeta_4 = \frac{2}{\sigma_{\rm min}(Q)} \left(\|C\|^2\|R\|+\|C\|\|B\|(\|C\|\|B\|+\zeta_1\gamma)\frac{\alpha}{\mu}\right).
\end{equation}
\end{lemma}
\begin{proof}
First, from \eqref{eq.derivative_P}, it is clear that 
\begin{equation} 
\begin{aligned}
 &\frac{\partial P_{\bar{K}}}{\partial \delta}=\sum_{j=0}^{\infty}{(A-B\bar{K}C)^\top}^j(C^\top {\Delta K}^\top E_{\bar{K}}\\
&\qquad \qquad \qquad \qquad +E_{\bar{K}}^\top \Delta KC)(A-B\bar{K}C)^j.
\end{aligned}
\end{equation}
Then, we can observe that
\begin{equation}
\begin{aligned}
C^\top {\Delta K}^\top &E_{\bar{K}}  + E_{\bar{K}}^\top \Delta KC \\
&=C^T\Delta K^TR\bar{K}C+C^T\bar{K}^TR\Delta K C\\
&\qquad-C^T\Delta K^TB^TP_{\bar{K}}(A-B\bar{K}C)\\
&\qquad-(A-B\bar{K}C)^TP_{\bar{K}}B\Delta KC\\
&\le (2\|C\|\|R\|\|\bar{K}C\|\|\Delta K\|\\
&\qquad+2\|C\|\|B\|\|P_{\bar{K}}\|\|A-B\bar{K}C\|\|\Delta K\|)I\\
&\le 2(\|C\|\|R\|q_3+\gamma\|C\|\|B\|\frac{\alpha}{\mu})\|\Delta K\|\frac{Q}{\sigma_{\rm min}(Q)},
\end{aligned}
\end{equation}
where the last step follows from Lemma \ref{lemma.K_upper_bound}.
Therefore, according to \eqref{eq.P_expand}, we have $\frac{\partial P_{\bar{K}}}{\partial \delta} \preceq \zeta_2 \|\Delta K\|P_{\bar{K}}$.

Next, we will prove that $\frac{\partial P'_{\bar{K}}[Z]}{\partial \delta} \preceq \zeta_3 \|\Delta K\|P_{\bar{K}}$. Based on \eqref{eq.P_derivative}, we get
\begin{equation}   
\frac{\partial P'_{\bar{K}}[Z]}{\partial \delta}= \sum_{j=0}^{\infty}{(A-B\bar{K}C)^\top}^jS_3(A-B\bar{K}C)^j,
\end{equation}
where
\begin{equation}  
\nonumber
\begin{aligned}
&S_3\\
&:=C^\top Z^\top (R+B^TP_{\bar{K}}B)\Delta KC-(BZC)^\top\frac{\partial P_{\bar{K}}}{\partial \delta}(A-B\bar{K}C)  \\
&\quad + C^\top \Delta K^\top  (R+B^TP_{\bar{K}}B) ZC-(A-B\bar{K}C)^\top\frac{\partial P_{\bar{K}}}{\partial \delta}BZC \\ 
&\quad-(B\Delta KC)^\top P'_{\bar{K}}[Z](A-B\bar{K}C)\\
&\quad-(A-B\bar{K}C)^\top P'_{\bar{K}}[Z]B\Delta KC.
\end{aligned}
\end{equation}
We can also show that
\begin{equation}
\nonumber
\begin{aligned}
S_3 &\le 2(\|C\|^2\|Z\|\|R+B^TP_{\bar{K}}B\|\|\Delta K\|\\
&\qquad+\|B\|\|Z\|\|C\|\|A-B\bar{K}C\|\|\frac{\partial P_{\bar{K}}}{\partial \delta}\|\\
&\qquad+\|B\|\|C\|\|A-B\bar{K}C\|\|P'_{\bar{K}}[Z]\|\|\Delta K\|)I\\
&\le 2(\|C\|^2\|R\|+\|C\|^2\|B\|^2\|P_{\bar{K}}\|+\zeta_2\gamma\|B\|\|C\|\|P_{\bar{K}}\|\\
&\qquad+\zeta_1\gamma\|B\|\|C\|\|P_{\bar{K}}\|)\|\Delta K\|I\\
&\le 2\big(\|C\|^2\|R\|+\|C\|\|B\|(\|C\|\|B\|\\
&\qquad\qquad\qquad+\zeta_1\gamma+\zeta_2\gamma)\frac{\alpha}{\mu}\big)\|\Delta K\|\frac{Q}{\sigma_{\rm min}(Q)}.
\end{aligned}    
\end{equation}
Therefore, we get $\frac{\partial P'_{\bar{K}}[Z]}{\partial \delta} \preceq \zeta_3 \|\Delta K\|P_{\bar{K}}$.

Similarly, for $P''_{\bar{K}}[Z]$, from \eqref{eq.second_derivative_P}, we can show that
\begin{equation}
\nonumber
\begin{aligned}
&S_1 \\
&\le 2(\|C\|^2\|Z\|^2\|R+B^TP_{K}B\|\\
&\qquad+\|B\|\|Z\|\|C\|\|A-BKC\|\|P'_{K}[Z]\|)I\\
&\le 2(\|C\|^2\|R\|+\|C\|^2\|B\|^2\|P_{K}\|+\zeta_1\gamma\|B\|\|C\|\|P_{K}\|)I\\
&\le 2\big(\|C\|^2\|R\|+\|C\|\|B\|(\|C\|\|B\|+\zeta_1\gamma)\frac{\alpha}{\mu}\big)\frac{Q}{\sigma_{\rm min}(Q)}.\\
\end{aligned}    
\end{equation}
So, it is clear $P''_{\bar{K}}[Z]\preceq \zeta_4 P_{\bar{K}} $, which completes the proof.
\end{proof}

\subsection{Proof of Lemma \ref{lemma.performance_diff}}
\label{append.proof}
\begin{proof}
Let ${x_t'}$ and ${u_t'}$ be the state and action sequences generated by $K'$, and $c_t'=x_t'^\top Qx_t'+u_t'^\top Ru_t'$. Then, one has
\begin{equation}
\nonumber
\begin{aligned}
J(K')-&J(K) \\
=& \mathbb{E}_{x_0\sim \mathcal{D}}\Big[\sum_{t=0}^{\infty}c_t' - V_K(x_0) \Big]\\
=& \mathbb{E}_{x_0\sim \mathcal{D}}\Big[\sum_{t=0}^{\infty}(c_t'+V_K(x_t')-V_K(x_t')) - V_K(x_0) \Big]\\
=& \mathbb{E}_{x_0\sim \mathcal{D}}\Big[\sum_{t=0}^{\infty}(c_t'+V_K(x_{t+1}')-V_K(x_t')) \Big],
\end{aligned}
\end{equation}
where the last step utilizes the fact that $x_0=x_0'$.

Let $A_K(x_t,K')=c_t+V_K(x_{t+1})-V_K(x_t)|_{u_t=-K'Cx_t}$, which can be expanded as 
\begin{equation}
\nonumber
\begin{aligned}
A_K&(x_t, K') \\
=& x_t^\top(Q + C^\top K'^\top RK'C)x_t\\
& +x_t^\top(A-BK'C)^\top P_K(A-BK'C)x_t - V_K(x_t)\\
=& x_t^\top(Q + (K'C-KC+KC)^\top R(K'C-KC+KC))x_t\\
& +x_t^\top(A-B(K'C-KC+KC))^\top P_K(A\\
&-B(K'C-KC+KC))x_t - V_K(x_t)\\
=&2x_t^\top(K'C-KC)^\top ((R+B^\top P_K B)KC-B^\top P_K A)x_t\\
& +x_t^\top(K'C-KC)^\top (R+B^\top P_K B)(K'C-KC)x_t\\
=&2x_t^\top(K'C-KC)^\top E_K x_t\\
& +x_t^\top(K'C-KC)^\top (R+B^\top P_K B)(K'C-KC)x_t.
\end{aligned}
\end{equation}
Then, we get that
\begin{equation}
\nonumber
\begin{aligned}
J(&K')-J(K) \\
=& \mathbb{E}_{x_0\sim \mathcal{D}}\Big[\sum_{t=0}^{\infty}A_K(x_t', K')  \Big]\\
=& \mathbb{E}_{x_0\sim \mathcal{D}}\Big[\sum_{t=0}^{\infty}\Big(2{\rm Tr}\big(x_t'x_t'^\top(K'C-KC)^\top E_K\big)+\\
&{\rm Tr}\big(x_t'x_t'^\top(K'C-KC)^\top (R+B^\top P_K B)(K'C-KC)\big)\Big)\Big]\\
=& 2{\rm Tr}\big(\Sigma_{K'}(K'C-KC)^\top E_K\big)+\\
&{\rm Tr}\big(\Sigma_{K'}(K'C-KC)^\top (R+B^\top P_K B)(K'C-KC)\big).
\end{aligned}
\end{equation}
\end{proof}

\subsection{Proof of Lemma \ref{lemma.gradient_dominance}}
\label{appendix.dominance}
\begin{proof}
Denote $\mathcal{L}_K=C\Sigma_{K}C^\top$. Let $X=(R+B^\top P_K B)^{-1}E_K \Sigma_{K'}C^\top\mathcal{L}_{K'}^{-1}$.
From Lemma \ref{lemma.performance_diff}, we find that
\begin{equation}
\nonumber
\begin{aligned}
J&(K')-J(K)\\
=& 2{\rm Tr}\big(\Sigma_{K'}(K'C-KC)^\top E_K\big)\\
&+{\rm Tr}\big(\Sigma_{K'}(K'C-KC)^\top (R+B^\top P_K B)(K'C-KC)\big)\\
=& {\rm Tr}\big(\Sigma_{K'}C^\top(\Delta K+X)^\top (R+B^\top P_K B)(\Delta K+X)C\big)\\
&-{\rm Tr}\big(\Sigma_{K'}C^\top\mathcal{L}_{K'}^{-1}C\Sigma_{K'}E_K^\top (R+\\
&\qquad \qquad \qquad \qquad B^\top P_K B)^{-1} E_K\Sigma_{K'}C^\top\mathcal{L}_{K'}^{-1}C\big)\\
\ge&-{\rm Tr}\big(\mathcal{L}_{K'}^{-1}C\Sigma_{K'}E_K^\top (R+ B^\top P_K B)^{-1} E_K\Sigma_{K'}C^\top\big),
\end{aligned}
\end{equation}
where $\Delta K=K'-K$  and the equality holds when $K'=K-X$.

Then, one has
\begin{equation}
\label{eq.gradient_dominance_12}
\begin{aligned}
&J(K)-J(K^*)\\
&\le {\rm Tr}\big(\mathcal{L}_{K^*}^{-1}C\Sigma_{K^*}E_K^\top (R+B^\top P_K B)^{-1} E_K\Sigma_{K^*}C^\top\big)\\
&\le\|\Sigma_{K^*}C^\top\mathcal{L}_{K^*}^{-1}C\Sigma_{K^*}\|{\rm Tr}\big( E_K^\top (R+B^\top P_K B)^{-1} E_K \big)\\
&\le\|\Sigma_{K^*}C^\top\mathcal{L}_{K^*}^{-1}C\|\|\Sigma_{K^*}\|{\rm Tr}\big( E_K^\top (R+B^\top P_K B)^{-1} E_K \big)\\
&\le \frac{\|\Sigma_{K^*}\|{\rm Tr}\big(E_K^\top E_K\big)}{\sigma_{\rm min}(R)}.
\end{aligned}
\end{equation}
From \eqref{eq.gradient}, it follows that 
\begin{equation}  
\label{eq.gradient_dominance_22}
\begin{aligned}
\|\nabla J(K))\|_F^2&= 4{\rm Tr}(C\Sigma_KE_K^\top E_K\Sigma_KC^\top)\\
&\ge 4\mu^2\sigma_{\rm min}(C)^2{\rm Tr}(E_K^\top E_K),\  \forall C \in \mathbb{C}.
\end{aligned}
\end{equation}

By \eqref{eq.gradient_dominance_12} and \eqref{eq.gradient_dominance_22}, it holds that 
\begin{equation}
\label{eq.gradient_dominance_proof}
J(K)-J(K^*)\le \frac{\|\Sigma_{K^*}\|\|\nabla J(K))\|_F^2}{4\mu^2\sigma_{\rm min}(C)^2\sigma_{\rm min}(R)}, \ \forall C\in \mathbb{C}.
\end{equation}

Suppose $K'$ satisfies that $K'=K-X$, then we get
\begin{equation}
\nonumber
\begin{aligned}
J(K)&-J(K^*)\\
&\ge J(K)-J(K')\\
&={\rm Tr}\big(\mathcal{L}_{K'}^{-1}C\Sigma_{K'}E_K^\top (R+B^\top P_K B)^{-1} E_K\Sigma_{K'}C^\top\big).
\end{aligned}
\end{equation}
Therefore, we further have 
\begin{equation}
\label{eq:lower_bound_22}
J(K)-J(K^*)\ge \frac{\mu{\rm Tr}\big(E_K^\top  E_K\big)}{\|R+B^\top P_K B\|}, \ \forall C \in \mathbb{C}.
\end{equation}

In addition, when $C\in \mathbb{C}$, since we can always identity the state $x$ by $x= (C^\top C)^{-1}Cy$, it is clear that $J(K^*)=J_s^*$ for every $C\in \mathbb{C}$. By replacing $J(K^*)$ in \eqref{eq.gradient_dominance_proof} and \eqref{eq:lower_bound_22} with $J_s^*$, we finally complete the proof.
\end{proof}

\bibliographystyle{ieeetr}
\bibliography{ref}

\end{document}